\documentclass[final,1p,times,authoryear]{elsarticle}
\usepackage{amsmath}
\usepackage{amsthm}
\usepackage{amssymb}
\usepackage{amsfonts}

\newtheorem{theorem}{Theorem}

\newcommand{\C}{\mathbb{C}}
\newcommand{\R}{\mathbb{R}}
\newcommand{\Q}{\mathbb{Q}}

\newcommand{\be}{\begin{equation}}
\newcommand{\ee}{\end{equation}}
\newcommand{\Z}{\mathbb{Z}}

\newcommand{\kxn}{k[x_1,\dots,x_n]}
\newcommand{\vv}{{\bf V}}
\newcommand{\la}{\langle}

\newcommand{\ra}{\rangle}

\newcommand{\mc}{\mathcal}
\begin{document}

\begin{frontmatter}

\title{Centers and limit cycles of a  generalized cubic Riccati  system}

\author[zz]{Zhengxin Zhou}
\address[zz]{School of Mathematical Sciences, Yangzhou University, Yangzhou 225002, P. R. China}
\ead{zxzhou@yzu.edu.cn}
\author[vr,vr1,vr2]{Valery G.~Romanovski}
\ead{valery.romanovsky@uni-mb.si}
\address[vr]{ 
Faculty of Electrical Engineering and Computer Science, University of Maribor,\\ Smetanova 17,  Maribor, SI-2000 Maribor, Slovenia}
\address[vr1]{CAMTP -
   Center for Applied Mathematics and Theoretical Physics,\\
University of Maribor, Krekova 2, SI-2000 Maribor, Slovenia}
\address[vr2]{Faculty of Natural Science and Mathematics, University of Maribor, Koro\v ska c. 160,  SI-2000 Maribor, Slovenia}
\author[ju]{Jiang Yu\corref{cor1}}
\ead{jiangyu@sjtu.edu.cn}
\address[ju]{School of Mathematical Sciences, Shanghai Jiao Tong University, Shanghai 200240,  P. R. China}

\begin{abstract}
We obtain condition for existence of a center for a cubic planar differential  system, which can be considered as a polynomial
 subfamily of the  generalized  Riccati system. We also investigate bifurcations of small limit cycles from the components of the center variety of the system.
\end{abstract}


\cortext[cor1]{Corresponding author}
\end{frontmatter}

\section{Introduction}

Consider systems of ordinary differential equations on $\R^2$ of the form
\begin{equation} \label{1}
\dot x = P(x,y),
\qquad
\dot y = Q(x,y),
\end{equation}
where $P$ and $Q$ are polynomials, $\max \{\deg P, \deg Q \} = n $. We view
\eqref{1} as defining a family of systems parametrized by the coefficients of
$P$ and $Q$.
That is,  the degree of polynomials $P$ and $Q$ in system
\eqref{1} is fixed and the coefficients of
the polynomials are parameters  $\lambda_1, \dots , \lambda_N$,
so the $N$-tuple of the parameters $p= (\lambda_1, \dots , \lambda_N)$
 is a point in  a   Euclidean
$N$-dimensional space ${\mc E}$,   and  we identify each point in  ${\mc E}$
with its corresponding system \eqref{1}.
  In this paper the space ${\mc E}$ of parameters is
either $\mathbb{R}^N$ or $\mathbb{C}^N$.

A singular point $(x_0,y_0)$ of  real system \eqref{1}
 is a stable  focus if  every  nearby trajectory
  spirals towards to it, an unstable focus if  every nearby
trajectory spirals away from it, and it is  a center if every nearby trajectory is an
oval.

A singular point $(x_0,y_0) \in \R^2$
of a system
$p \in {\mc E}\subset \mathbb{R}^N $ is said to have \emph{cyclicity} $k$ with respect
to ${\mc E}$
if and only if any sufficiently small perturbation of $p$ in ${\mc E}$
has at
most $k$ limit cycles in a sufficiently small neighborhood of $(x_0,y_0)$, and
$k$
is the smallest number with this property.
The problem of  cyclicity of
a center or a focus of a system of the form
\eqref{1}  is known as
the \emph{local 16th Hilbert problem} in
(\cite{FY}), based on its connection to
Hilbert's still unresolved 16th
problem, which in part asks for a bound on the
number of limit cycles of  system
\eqref{1} in terms of degree  $n$ of the system  (see, e.g.  survey papers  \cite{G,JBL}).

The  problem of cyclicity is closely connected to the center problem,
that is, the problem of finding systems with centers in a given   polynomial
family  \eqref{1}. The studies of the   center problem dates back to 1908
when  \cite{Dul} investigated the case of the quadratic system.
 The literature
devoted to the subject is vast,  see e.g.  \cite{A-L-S, CLi,  LLH, RS,  Sib}  and the references that they contain.

After a linear  transformation and time rescaling any  planar polynomial system
 \eqref{1} with an elementary  center or weak  focus   can be written in the form
\begin{equation} \label{sys introd 1}
\dot{x} = -y +\widetilde  P(x,y),
\hspace{0.3cm}
\dot{y} = x +\widetilde  Q(x,y),
\end{equation}
where $\widetilde  P$ and $\widetilde Q$ are  polynomials without free and linear terms.
By the Poincar\'e-Lyapunov  theorem
for real  system \eqref{sys introd 1}
the origin is a center  if and only if in a neighborhood of the origin  the system  admits a real analytic local first
integral of the form
\be \label{Int}
\Psi(x,y) = x^2+y^2  + \sum_{j+k = 3} \psi_{jk} x^j y^k.
\ee

Recently
  \cite{Lli2014, Lli2015}   investigated    the
planar    system
\begin{equation} \label{sys-Ric2}
\dot{x} =  f(y),~~~~
\dot{y} = g_2(x)y^2 + g_1(x)y + g_0(x),
\end{equation}
which is called the
 generalized Riccati system \eqref{sys-Ric2}, since it becames
the classical   Riccati system if $f( x) \equiv  1$.

In this paper we  consider a particular subfamily  of   the   generalized Riccati system, namely, the cubic  system
\begin{equation} \label{sys-Ric3}
\begin{aligned}
\dot{x}& =     -y + a_{02} y^2,
\\
\dot{y} &= ( b_{02} +
    b_{12} x  )y^2 + (  b_{11} x+b_{21} x^2 )y+ (b_{20} x^2  +b_{30} x^3+x)  \\
& =  x + b_{20} x^2 + b_{11} x y + b_{02} y^2 +  b_{30} x^3 + b_{21} x^2 y+
    b_{12} x y^2,
  \end{aligned}
\end{equation}
where   $a_{ij}, b_{ks} $ are real or  complex parameters. We first find
 conditions for system \eqref{sys-Ric3} with
complex parameters  to have an analytic
first integral of the form   \eqref{Int}, so, the corresponding real systems
have a center at the origin. Then, we study limit cycles bifurcations from the centers
of real systems \eqref{sys-Ric3}.

\section{Preliminaries}


For system \eqref{sys introd 1}  it is always possible to
find a function
 \be \label{Intf}
\Phi(x,y) = x^2+y^2  + \sum_{j+k = 3}^\infty \phi_{jk} x^j y^k,
\ee
such that
\be \label{phim}
\dfrac{\partial \Phi }{\partial x} (-y +\widetilde  P(x,y)) +
\dfrac{\partial \Phi}{\partial y} (x+\widetilde  Q(x,y))
  = \sum_{m=2}^\infty   g_{m-1} (u^2+v^2)^m.
\ee
The coefficients $g_k$ in \eqref{phim} are polynomials in parameters of system
\eqref{phim} called the  \textit{focus quantities} of the system.
Each polynomial $g_i$ represents an obstacle for existing of integral \eqref{Int},
that is, system \eqref{sys introd 1} admits an integral \eqref{Int} if and only if
$g_1=g_2=g_3= \dots =0$.
Thus, the set of all systems in the parametric family \eqref{sys introd 1} with centers
 (equivalently, the set of systems with a first integral
of the form \eqref{Int}) is the variety\footnote{  We remind that the  variety   of a given  ideal  $F$
generated by polynomials   $ f_1(x_1,\dots,x_m),\dots,$  $ f_s(x_1,\dots,
x_m) $  over a field $k$,
$F=\la f_1,\dots, f_s\ra\subset k[x_1,\dots,x_m], $
is the set
$
{\bf V}(F)=\{(x_1,\dots,x_n)\in k^n |\ f(x_1,\dots,x_m)=0  \ {\rm
for \ all \ } f\in F  \}. $ }
 $V^{\mathbb R} \subset \mathbb{R}^n $ of the ideal $B=\la g_1, g_2, g_3,\dots \ra  $.
By the Hilbert Basis Theorem there is an integer $k$ that $B=\la g_1, \dots, g_k \ra$,
however there are no regular methods  to find such $k$.

We will also  consider system \eqref{sys introd 1}  as a system with complex
parameters. In such case the polynomials $g_i$ in \eqref{phim} are
polynomials with complex coefficients and then  their  variety is
a complex variety, which we will denote $V^{\mathbb C}$. All systems whose parameters
belong to the variety $V^{\mathbb C}$ admit local analytic first integral of the form \eqref{Int}. We call   $V^{\mathbb R}$ and   $V^{\mathbb C}$ the real and complex center varieties of system   \eqref{sys introd 1}.

It was shown in \cite{GLM} that knowing the   complex variety of a real polynomial system  can be helpful for investigation
of cyclicity of the  real system. We also need to work with complex varieties
to apply our computational approach for computing conditions of integrability.

To find a real or complex center variety of a polynomial
system \eqref{sys introd 1}, one computes a few first focus
quantities $g_1,\dots, g_k$  of the system and then finds decomposition
of the variety of the  ideal $ B_k=  \la  g_1,\dots, g_k\ra $
obtaining the necessary conditions of existence of integral
\eqref{Int}. Then it is necessary to prove the sufficiency of the
obtained condition. We recall that  one of most powerful
methods to prove the existence of integral \eqref{Int} is the
Darboux method, which allows to construct a
first integral or integrating factor using Darboux polynomials (see, e.g. surveys \cite{Llibre,LZ}).
A \textit{Darboux polynomial} of system \eqref{1} is a polynomial $f(x,y)$ satisfying
\begin{equation*}
\dfrac{\partial f}{\partial x} P +
\dfrac{\partial f}{\partial y} Q= K f,
\end{equation*}
where $K(x,y)$ is a polynomial called the \textit{cofactor of $f$}.
It is easy to see that if $f$ is  a Darboux polynomial of system \eqref{1},
then $f=0$ is an algebraic invariant curve of the system.

It is easy to see that if  system \eqref{sys introd 1} has $p$ irreducible
Darboux polynomials $f_1, ... , f_p$ with the associated cofactors $K_1, ... , K_p$,
such that
$
s_1K_1+...+s_pK_p = 0, $
then  the  function
$
 H = f_1^{s_1}... f_p^{s_p}
 $
 is a first integral of the  system
(called the  Darboux first integral)
and if
$$
s_1K_1+...+s_pK_p = \frac{\partial P}{\partial x}+ \frac{\partial Q}{\partial y},
$$
then the  function
$
 \mu  = f_1^{s_1}... f_p^{s_p}
 $
 is an integrating factor of \eqref{1}
(called the Darboux integrating factor).

\section{Itegrability conditions for system \eqref{sys-Ric3}}

In this section we assume that parameters of system
\eqref{sys-Ric3} are complex and find  the complex variety $V^{\mathbb C}$ of the system.
The study yields that the variety consists of 7 irreducible components.

\begin{theorem} \label{th1}
System \eqref{sys-Ric3} has an analytic integral of the form \eqref{Int}  if
the 7-tuple  $ (a_{02}, b_{20}, b_{11},  b_{12}, b_{02},$ $ b_{21}, b_{30})$ of
its parameters belong to the variety of one of the following prime  ideals: \\

\noindent $
I_1= \la  b_{21}, b_{20}, b_{02}\ra,
\\
I_2= \la b_{30},    b_{12},    b_{02},   b_{11}b_{20}-b_{21}
\ra, \\
I_3= \la  b_{30}, b_{21}, b_{12}, -2b_{02}b_{11}^2 + 4b_{02}^2b_{20} - b_{11}^2b_{20},
 2a_{02}b_{11} + b_{11}^2 - 4b_{02}b_{20}, 2a_{02}b_{02} - b_{02}b_{11} - b_{11}b_{20},
 4a_{02}^2 - b_{11}^2 - 4b_{20}^2
\ra, \\
I_4= \la b_{21}, b_{11},a_{02}
\ra, \\
I_5= \la a_{02}, b_{02}b_{21} + b_{11}b_{30}, 2b_{02}b_{12} + b_{12}b_{20} + b_{02}b_{30}, b_{02}b_{11} + b_{11}b_{20} - b_{21},
 b_{02}^2 + b_{02}b_{20} + b_{30}, b_{12}b_{20}b_{21} - 2b_{11}b_{12}b_{30} - b_{11}b_{30}^2,
 b_{11}b_{20}b_{21} - b_{21}^2 - b_{11}^2b_{30}, b_{12}b_{20}^2 - 4b_{12}b_{30} - b_{02}b_{20}b_{30} - 2b_{30}^2,
 b_{11}b_{12}b_{20} - 2b_{12}b_{21} + b_{11}b_{20}b_{30} - b_{21}b_{30},
 -(b_{12}b_{21}^2) + b_{11}^2b_{12}b_{30} + b_{11}^2b_{30}^2
\ra, \\
I_6= \la b_{21},b_{12},b_{11},b_{02}
\ra, \\
I_7= \la  b_{21}, b_{12},b_{30}, 3b_{02}+5b_{20},5a_{02}-b_{11}, 6b_{11}^2+25b_{20}^2 \ra.
$

That is, the variety $V^{\mathbb C}$ of system has the irreducible
decomposition $V^{\mathbb C}=\cup_{k=1}^7 V_k^c$, where $V_k^c=\vv(I_k)$ ($k=1,\dots, 7$).

\end{theorem}

\begin{proof}

Necessity.
For system \eqref{sys-Ric3}
we have computed eight first focus quantities
$g_1, g_2, \dots, g_8$ and then tried to find
the irreducible decomposition of the variety
$\vv(I)$ of the ideal
\be \label{iG}
I=\la g_1, g_2, \dots, g_8 \ra
\ee
over the field of rational numbers using the routine
\texttt{minAssGTZ} \cite{D-L-P-S} (which is based on the algorithm of \cite{GTZ})
  of the computer algebra system {\sc Singular} \cite{sing}, but  due to high
complexity of calculations we have not succeeded
to complete  them with our computational facilities.
However computing in the polynomial
 ring 
 $$\Z_{32003}[a_{02}, b_{02}, b_{11},  b_{12}, b_{20}, b_{21}, b_{30}]\footnote{It appear, for the first time 
  modular computations were used for studies on the center and cyclicity problems in \cite{Ed}}$$ using the degree reverse lexicographic ordering
with  $ a_{02} > b_{02} >  b_{11} > b_{12} > b_{20} >  b_{30} >b_{21}$
 we have found that in the affine space  $\Z_{32003}^7$
 the variety of $I$ consists of 8 irreducible components
 defined by the following 8 prime ideals:
\\
$
 \tilde I_1=\la  b_{21}, b_{20}, b_{02} \ra,
\\
\tilde I_2= \la b_{30},    b_{12},    b_{02},   b_{11}b_{20}-b_{21} \ra,
\\
\tilde I_3=\la  b_{30}, b_{21}, b_{12},
   a_{02}b_{11}-16001 b_{11}^2-2b_{02}b_{20},
   a_{02}b_{02}+16001 b_{02}b_{11}+16001b_{11}b_{20},
   a_{02}^2-8001 b_{11}^2-b_{20}^2,
   b_{02}b_{11}^2-2b_{02}^2b_{20}-16001 b_{11}^2b_{20}\ra,
\\
\tilde I_4= \la b_{21}, b_{11},a_{02}\ra,
\\
\tilde I_5= \la a_{02},
   b_{02}b_{21}+b_{11}b_{30},
   b_{02}b_{12}-16001 b_{12}b_{20}-16001 b_{02}b_{30},
   b_{02}b_{11}+b_{11}b_{20}-b_{21},
   b_{02}^2+b_{02}b_{20}+b_{30},
   b_{12}b_{20}b_{21}-2b_{11}b_{12}b_{30}-b_{11}b_{30}^2,
   b_{11}b_{20}b_{21}-b_{11}^2b_{30}-b_{21}^2,
   b_{12}b_{20}^2-b_{02}b_{20}b_{30}-4b_{12}b_{30}-2b_{30}^2,
   b_{11}b_{12}b_{20}+b_{11}b_{20}b_{30}-2b_{12}b_{21}-b_{21}b_{30},
   b_{11}^2b_{12}b_{30}+b_{11}^2b_{30}^2-b_{12}b_{21}^2 \ra,
\\
\tilde I_6=\la  b_{21},b_{12},b_{11},b_{02} \ra,
\\
\tilde I_7= \la
   b_{30},
   b_{21},
   b_{12},
   b_{11}-15273 b_{20},
   b_{02}-10666 b_{20},
   a_{02}+3346 b_{20}\ra,
   \\
\tilde I_8=\la
   b_{30},
   b_{21},
   b_{12},
   b_{11}+15273 b_{20},
   b_{02}-10666 b_{20},
   a_{02}-3346 b_{20}\ra.
$

Lifting these ideals into the ring
 $\Q[a_{02}, b_{02}, b_{11},  b_{12}, b_{20}, b_{21}, b_{30}]$
using the rational reconstruction algorithm of \cite{WGD}
we obtain  the ideals $I_1, \dots, I_6$ given in the statement
of Theorem \ref{th1} and
the ideals
\\
$ \hat I_7 =\la
 b_{30}
 ,
b_{21}
,
b_{12}
,
b_{11} + \frac{51}{44} b_{20}
,
b_{02} + \frac{5}{3}  b_{20}
,
a_{02} +\frac{161}{67}  b_{20}
 \rangle
$\\
 and
\\
$ \hat I_8 =\la
 b_{30}
,
b_{21}
,
b_{12}
,
b_{11} - \frac{51}{44}  b_{20}
,
b_{02} + \frac{5}{3} b_{20}
,
a_{02} -\frac{161}{67}  b_{20}
 \rangle
$.

To check the correctness of the obtained decomposition
we use the procedure proposed in  \cite{RP}.

First, using the Radical Membership Test\footnote{
The test  says that  for  a polynomial $f$ and an ideal $I=\la f_1, \dots, f_m\ra $ in $\C[x_1,\dots,x_n]$  $f|_{\vv(I)}\equiv 0$ if and only
if the reduced Gr\"obner basis of the ideal $\la 1- w f, f_1, \dots,f_m \ra $ (here $w$ is a new variable) is equal to $\{ 1\}$,  see e.g. \cite{Cox,RS}
for more details}
 we check if all focus quantities
$g_i$ ($i=1,\dots,8$) vanish on each of
the varieties  $\vv(I_1), \dots, \vv (I_6), \vv(\hat I_7),
 \vv(\hat I_8)$.
   The calculations show that all polynomials
 $g_i$ are equal to zero on each of
 varieties  $\vv(I_1), \dots, \vv (I_6)$,
 but not on the varieties $ \vv(\hat I_7),
 \vv(\hat I_8)$. This means, that $ \vv(\hat I_7)$ and $
 \vv(\hat I_8)$ are not correct
  components of the irreducible decomposition
  of $\vv(I)$. A usual recipe to find  the
  correct components of the   decomposition in such situation
  is to recompute the  decomposition over
  a few fields of larger characteristics (see e.g. \cite{EA}).
  However instead of doing this we observe that
  both ideals $\tilde I_7 $ and   $\tilde I_8 $
  contain the polynomials $ b_{30}, b_{21}, b_{12}$
  and compute  with   \texttt{minAssGTZ} of {\sc Singular}
 the minimal associate primes of
 the ideal $\la I, b_{30}, b_{21}, b_{12} \ra $ over the field $\mathbb{Q}$
  obtaining
the component defined by the
  ideal $I_7$ of the statement of the   theorem.

Now, to check the correctness of the obtained
conditions  we  computed the ideal
$ \tilde I =   \cap_{s=1}^7 I_s  $, which defines the union
of all seven   components listed in the statement of  the theorem and
 have  checked  that reduced   Gr\"obner bases of all
ideals $\la \tilde{I}, 1-w g_k \ra$ (where $k=1,\dots, 8$ and $w$ is a new variable) computed over $\mathbb{Q}$
are $\{1\}$. By the Radical Membership Test  it means that
\be \label{invc2}
\vv(\tilde I)\subset \vv (I).
\ee
  To check the opposite inclusion
it is sufficient to check that
 \be  \label{clm}
 \la I, 1-w f \ra =\la 1 \ra
 \ee
for  all polynomials $f$ from a basis
of     $ \tilde{I}$. Unfortunately,  we were not able
to perform the check over the field   $\mathbb{Q}$
however we have checked that \eqref{clm}
holds over a  few fields of finite characteristic.
It yields that \eqref{clm} holds with high probability \cite{EA}.

  Sufficiency. We now prove that if
  the coefficients of the system belong to one of
  varieties mentioned in the statement of the theorem
  then the system has an analytic first integral of the form \eqref{Int}.

Usually the center variety of a polynomial system
contains components corresponding to Hamiltonian, time-reversible
and Darboux integrable systems.

It is easy to see that systems from $V^{\mathbb C}_6$
are Hamiltonian with the Hamiltonian
$$
 H=\frac{ x^2 + y^2}2+ \frac{b_{20} x^3}3 + \frac{b_{30} x^4}4  - \frac{a_{02} y^3}3.
  $$

All time-reversible cubic  systems were found
in \cite{Sib,JLR}.  To use the results of    \cite{Sib,JLR}
we first complexify system \eqref{sys-Ric3} introducing
the variable $z=x+i y$ and  obtain from \eqref{sys-Ric3} after rescaling of time by $i$  the complex
differential equation\\
\begin{multline*}
\dot z =z  + \frac{1}{4}z ^2( i a_{02} - b_{02} -  i b_{11} + b_{20}) - \frac{1}{2}z \bar{z }( i  a_{02} - b_{02}- b_{20}) +\frac{1}{4}\bar{z }^2(  i  a_{02} - b_{02} + i  b_{11} + b_{20})
\\  -\frac{1}{8}z ^3 (b_{12} + i b_{21} - b_{30})   +  \frac{1}{8} z ^2\bar{z }(b_{12} - i  b_{21}+ 3 b_{30}) +\frac{1}{8}z \bar{z }^2( b_{12} + i  b_{21} +3 b_{30})
  -\frac{1}{8}\bar{z }^3 (b_{12}- i  b_{21}- b_{30}).
\end{multline*}

Substituting the coefficients of this differential equations
into polynomials of Theorem 6 of \cite{JLR}, which define
the variety of all time-reversible cubic systems, and
then computing   with
\texttt{minAssGTZ} of {\sc Singular} the minimal associate primes of the obtained ideal    we get  the components
$V_1^c, V_3^c$ and $V_4^c$ of  the statement of the theorem.
Hence,  all systems from the components $V_1^c, V_3^c$ and $V_4^c$
are time-reversible and, therefore, admit a first integral
of the form \eqref{Int} (see e.g. \cite{RS} for more details).

Thus, there remains to prove integrability of systems
from the components $V_2^c$, $V_5^c$ and $V_7^c$.

Systems from the component  $V_2^c$,  have the form
 \be \label{sc1}
  \dot x=   -y + a_{02} y^2, \qquad
\dot y=
x + b_{20} x^2 + b_{11} x y + b_{11} b_{20} x^2 y.
   \ee
   Computing we obtain that system \eqref{sc1}
    admits the Darboux polynomial
   $f=1+b_{11} y$ with the cofactor
   $ K=
   b_{11} x (1 + b_{20} x).
  $
  Thus,
  the function $\mu= \frac{1}{f}$ is a Darboux integrating
  factor of \eqref{sc1}, which allows
  to construct the analytic first integral
  $$
  \Psi =\frac{1}{3}(\frac{3 b_{11} y(2(a_{02}+b_{11})-a_{02}b_{11}y)-6(a_{02}+b_{11})\log(b_{11}y+1)}{b_{11}^3}+x^2(2b_{20}x+3))
  =x^2+y^2+\dots.
  $$

 Systems from $V^{\mathbb C}_5$ are the so-called reduced   Kukles systems.
The center problem for such systems has been solved in \cite{JW,CL},
so by the results of these papers systems from   $V^{\mathbb C}_5$ admit first integral
of the form \eqref{Int}.

 Finally, system  from the component $V^{\mathbb C}_7$ are of the form
 \be
 \dot x = -y\pm\frac{i b_{20} y^2}{\sqrt{6}},\qquad
 \dot y=   x+b_{20} x^2\pm \frac{5 i b_{20} x y}{\sqrt{6}}-\frac{5b_{20}y^2}{3}
 \ee
The system has
the Darboux polynomial
$$
f=1+\frac{b_{20}^2 x^2}{3}-
\frac{b_{20}^2 y^2}{18}+\frac{4b_{20}x}{3}\pm   \frac{1}{3}i\sqrt{\frac{2}{3}}b_{20}^2 xy\pm   i\sqrt{\frac{2}{3}} b_{20} y
$$
with the cofactor
$
K=\pm \frac{1}{3}ib_{20}(\sqrt{6}x+4iy)
$
 which allows to construct the integrating factor
 $
 \mu= f^{-5/2}.
$
\end{proof}

{\it Remark.} In the statement of theorem the ideals are presented as returned by the
routine \texttt{ minAssGTZ} of {\sc Singular}. However looking for
Gr\"obner bases of $I_5$ with different ordering of variables we find that the ideal
$I_5$ is the same as the ideal
\be \label{hati5} \hat I_5=\la b_{02}^3+b_{02}^2 b_{20}-2 b_{02} b_{12}-b_{12} b_{20}
   , -b_{02} b_{11}-b_{11} b_{20}+b_{21}
   , b_{02}^2+b_{02} b_{20}+b_{30} \ra.
   \ee  It is easy to see that the conditions defined
   by these polynomials are equivalent to conditions (iv) of Theorem 3.6 of \cite{CL}.

\section{Limit cycle bifurcations in system \eqref{sys-Ric3}}

In this section we study bifurcations of limit cycles from
each  component of the real  center variety of system \eqref{sys-Ric3} under perturbations inside the family.
It is obvious that $V^{\mathbb C}_7$ is the empty set in $\mathbb{R}^7$.
So the real variety $V^{\mathbb R}$ of \eqref{sys-Ric3} consist of 6 components,
$V^{\mathbb R}=\cup_{k=1}^6 V_k$, where $V_k$ is the set $V^{\mathbb C}_k$ of Theorem
\ref{th1} restricted to $\R^7$.

Let $I=\la f_1, \dots, f_m\ra \subset \kxn $ be an ideal
and $\vv(I)$ be its variety,
 Assume that
a  decomposition of $V=\vv(I)$ is known  and
 let $p$ be a point from $V$. The tangent space to $V$ at $ p$
is defined as
$
T_{p}=p+\left\{v| J_p(I)v=0\right\},
$
where $J(I)$ is the Jacobian of the polynomials $f_1,\dots,f_m$
and $J_p$ indicates that it is evaluated at $p$.
It follows that $\dim
T_{p}=n-rank ( J_p(I)).
$
It is said that $p$ is a smooth point of $V$ if $\dim
T_{p}=\dim V_p.$ Let $C$ be a component of $V$ of codimension $k$
and assume that
$p\in C$, $rank (J_p(I))=s$. Then $k\ge s$ and $p$ is a smooth
point $C$ if and only if $k=s$; in this case $rank (J_q(I))=k$
 at any smooth point of $C$.


The following statement is proved by 
 \cite{Ch}.
\begin{theorem} \label{CT}
Assume that for system \eqref{sys introd 1} $p  \in  K$
 is a point on the center variety and that the
first $k$ of the focus quantities  $g_i$ have independent linear parts.
Then $p$ lies on a component of the center variety  of codimension
at least $k$ and there are bifurcations of \eqref{sys introd 1}
which produce $k$ limit cycles locally from the center
corresponding to the parameter value $p$.

If, furthermore, we know that $p$  lies on a component of the center variety of
codimension $k$, then $p$ is a  smooth point of the variety, and the cyclicity of the center
for the parameter value $p$  is exactly $k-1$.

In the latter case, $k - 1$ is also the cyclicity of a generic point on this component of the center variety.
\end{theorem}

According to the theorem in some cases the cyclicity of generic point of a component of the
center variety can be easily determined if we know the dimension
of the components of  center variety. The dimension of a complex  variety can be
 computed using algorithms of computational algebra, since
it is equal to the degree of the  affine Hilbert polynomial of any ideal
defining the variety. However determining  dimensions of
real varieties is more difficult problem.  Nevertheless, it is not difficult
to determine the dimension of the components of the center variety of   real system \eqref{sys introd 1}
and to prove the following result.

\begin{theorem}
 \label{t5}
 The cyclicities of generic point of the components $V_1,V_2,V_4, V_5,  V_6$
 of system \eqref{sys-Ric3} are 2,3, 2, 3, 3  respectively.
 The cyclicity of generic point of $V_3$ is at least 2.
\end{theorem}
\begin{proof}
It is clear that the codimension of the component
$V_1$ is 3. Computing the minors of the matrix  $J( g_1,g_2, g_3)$
evaluated on $V_1$
we see that for any point of $p\in V_1$ there is a non-zero 3-minor if $f(p)\ne 0$,
where

$
f=-140 a_{02}^6 b_{12} + 61 a_{02}^4 b_{11}^2 b_{12} + 9 a_{02}^3 b_{11}^3 b_{12} -
  2 a_{02}^2 b_{11}^4 b_{12} - 36 a_{02}^4 b_{12}^2 + 9 a_{02}^2 b_{11}^2 b_{12}^2 +
  600 a_{02}^5 b_{11} b_{30} + 450 a_{02}^4 b_{11}^2 b_{30} +
  45 a_{02}^3 b_{11}^3 b_{30} - 15 a_{02}^2 b_{11}^4 b_{30} -
  100 a_{02}^4 b_{12} b_{30} + 64 a_{02}^3 b_{11} b_{12} b_{30} +
  37 a_{02}^2 b_{11}^2 b_{12} b_{30} + 11 a_{02} b_{11}^3 b_{12} b_{30} +
  12 a_{02} b_{11} b_{12}^2 b_{30} - 6 b_{11}^2 b_{12}^2 b_{30} +
  750 a_{02}^3 b_{11} b_{30}^2 + 240 a_{02}^2 b_{11}^2 b_{30}^2 -
  15 a_{02} b_{11}^3 b_{30}^2 - 90 a_{02} b_{11} b_{12} b_{30}^2 +
  18 b_{11}^2 b_{12} b_{30}^2 + 210 a_{02} b_{11} b_{30}^3.
$.

Therefore, by Theorem \ref{CT} the cyclicity of a generic point
of $V_1$ is two (more precisely, the cyclicity of point $p\in V_1$
is two  if $f(p)\ne 0$).

Similar consideration and conclusion  are  valid for the  component $V_4$.

It is also obvious that the codimensions of the components  $V_2$
and $V_6$ is 4 and the computations show that the rank of  $J( g_1,\dots, g_4)$ at  a generic point of the component is 4, so by Christopher's theorem the cyclicity of generic
point of the component is 3.

To find dimensions of components $V_3$ and $V_5$ we
look for their parametrizations. The component
$V_5$ can be parametrized as follows:
\be \label{par3}
a_{02}= 0,\ b_{20}=  -\frac{t_2
   \left(t_2^2-2
   t_3\right)}{t_2^2-t_3},\
   b_{30}=
   -\frac{t_2^2
   t_3}{t_2^2-t_3},\
   b_{11}= t_1,\
b_{21}=
   \frac{t_1 t_2
   t_3}{t_2^2-t_3},\
b_{02}=
   t_2, b_{12}= t_3.
\ee
To check this
we eliminate from the ideal
\begin{multline*}
\la 1-w ( t_2^2-t_3 ),a_{02},  b_{20}+\frac{t_2
   \left(t_2^2-2
   t_3\right)}{t_2^2-t_3},\
   b_{30}+ \frac{t_2^2
   t_3}{t_2^2-t_3},\
   b_{11}- t_1,\
b_{21}-
   \frac{t_1 t_2
   t_3}{t_2^2-t_3},\
b_{02}-
   t_2, b_{12}- t_3\ra
\end{multline*}
   of the ring $  \R[w,t_1,t_2,t_3, a_{02},  b_{20},
   b_{30},
   b_{11},
b_{21}, b_{12},
b_{02} ]$
 the variables $w,t_1,t_2,t_3$  obtaining
 the ideal $I_5$ from the statement of Theorem \ref{th1}.
By Theorem 2 of \cite[\S 3.3]{Cox} it
  means that \eqref{par3} gives a rational parametrization
of the variety $\vv(I_5)$. Now it is easy to conclude
that
the dimension of the component is 3 and, hence,
the codimension is 4. Checking that the rank of  $J( g_1, \dots, g_4)$ is
four almost everywhere on $V_5$ we obtain by Theorem \ref{CT} that the cyclicity
of  generic point of the component is 3.

A parametrization of $V_3$ is given by
$$
a_{02}= -\frac{t_1 \left(t_1^2+4
   t_2^2\right)}{2 (t_1-2
   t_2) (t_1+2
   t_2)},\
b_{20}=
   \frac{2 t_1^2
   t_2}{4
   t_2^2-t_1^2},\
b_{30}=
   0,\
b_{11}=
   t_1, b_{21}=0, b_{02}=
   t_2,  b_{12}=0.
$$
From the parametrization
we see that the component is two-dimensional,
thus, it is of codimension five.
However only first three focus quantities have
independent linear parts, so we cannot get an upper
bound for cyclicity using Theorem \ref{CT} and only conclude
that the cyclicity is at least two.
\end{proof}

{\it Remark.} Of course, it is often happens that there  many parametrizations of the same  variety.
Another parametrization of $V_5$ can be easily obtained using the  polynomial basis of $I_5$  presented in \eqref{hati5}.

\section*{Acknowledgments}

The first author is supported  by the NSF of China under grant  11571301 and the NSF of province Jiangsu under grant  BK20161327. The second author  is  supported by the Slovenian Research Agency (ARRS, program P1-0306). The third author is partially supported by the state key program of NNSF of China grant number 11431008,  NSF of Shanghai grant number 15ZR1423700.

\bibliographystyle{elsarticle-harv}

\end{document}